\documentclass[11pt,reqno]{amsart}
\usepackage{amsmath,amssymb}
\usepackage{graphicx,psfrag}
\usepackage{amsthm}
\usepackage{enumerate,placeins,wrapfig,booktabs,chngcntr}
\usepackage{subcaption}

\usepackage[colorlinks=true, linkcolor=blue, citecolor=green]{hyperref}
\newtheorem{thm}{Theorem}[section] 
\newtheorem{conj}[thm]{Conjecture}
\newtheorem{qus}[thm]{Question}

\newtheorem{cor}[thm]{Corollary}
\newtheorem{lem}[thm]{Lemma}
\newtheorem{prop}[thm]{Proposition}

\newtheorem{defn}[thm]{Definition}

  \allowdisplaybreaks[4]

\begin{document}

    \title[Absence of periodic points]{Accessibility and Central Integrability In the Absence of Periodic Points}

	\author{Ziqiang Feng}
    \address{Beijing International Center for Mathematical Research, Peking University, Beijing, 100871, China.}
    \email{zqfeng@pku.edu.cn}

    \author{Ra\'{u}l Ures}
    \address{Department of Mathematics and Shenzhen International Center for Mathematics, Southern University of Science and Technology, Shenzhen, 518000, China.}
    \email{ures@sustech.edu.cn}

    \thanks{This work was partially supported by National Key R\&D Program of China 2022YFA1005801, NSFC 12071202, and NSFC 12161141002.}

	\begin{abstract}
        We consider a partially hyperbolic diffeomorphism $f: M \to M$ without periodic points on a closed manifold $M$. We prove that $f$ is accessible when $M$ is a 3-manifold with non-virtually-solvable fundamental group $\pi_1(M)$. In the case where $\dim E^c = 1$, we demonstrate that the center bundle $E^c$ is uniquely integrable if $f$ lacks accessibility. Furthermore, we provide a complete characterization of accessibility classes for such systems with one-dimensional center bundles.
	\end{abstract}

    \date{November 2, 2025}

	\maketitle


	\vspace{.5cm}



\tableofcontents

\section{Introduction}

The primary focus of this paper is the accessibility property in partially hyperbolic systems and its interplay with central integrability. A diffeomorphism $f: M \to M$ on a closed Riemannian manifold $M$ is \emph{partially hyperbolic} if the tangent bundle $TM$ splits into three $Df$-invariant subbundles $E^s \oplus E^c \oplus E^u$, where $E^s$ is uniformly contracting, $E^u$ is uniformly expanding, and $E^c$ is dominated by the other two subbundles. For a precise definition, see Section \ref{PH}. A partially hyperbolic diffeomorphism is \emph{accessible} if any two points in $M$ can be connected by a piecewise smooth path tangent to either $E^s$ or $E^u$.  

The study of partial hyperbolicity originated in the work of \cite{BrinPesin74} on skew products and frame flows, and \cite{PughShub72} on Anosov actions. Accessibility was introduced in \cite{BrinPesin74} as a generalization of Anosov systems, where it termed as transitivity of a pair of foliations, to investigate the topological transitivity of partially hyperbolic dynamics. In 1995, Pugh and Shub proposed a conjecture, as a part of their Stable Ergodicity Conjecture, that accessibility is an open and dense property among partially hyperbolic diffeomorphisms.

\begin{conj}\label{PS}(Pugh-Shub)  
    Accessibility holds for an open and dense subset of partially hyperbolic diffeomorphisms.  
\end{conj}  

Significant progress has been made toward proving this conjecture. Didier \cite{Didier} established the $C^1$-stability of accessibility for partially hyperbolic diffeomorphisms with one-dimensional center bundle. Hertz, Hertz, and Ures \cite{08invent} demonstrated $C^{\infty}$-denseness of accessibility for conservative systems with a one-dimensional center bundle. The work of \cite{BHHTU08} extended this result to the non-conservative setting, thereby resolving Conjecture \ref{PS} for systems with a one-dimensional center bundle.  

For two-dimensional center bundles, Avila and Viana \cite{AvilaViana20_2dcenter} established $C^1$-openness of accessibility among $C^2$ partially hyperbolic diffeomorphisms and $C^r$-denseness for certain fibered systems. J. R. Hertz and Vásquez \cite{HertzVasquez20} proved that accessibility classes form immersed submanifolds for systems with two-dimensional center bundles. Leguil and Pineyrua \cite{LeguilPineyrua21} recently contributed to $C^r$-denseness results for systems with two-dimensional center bundles under strong bunching and plaque expansive conditions. Besides, Dolgopyat--Wilkinson \cite{DW03} and Avila--Crovisier--Wilkinson \cite{ACW22plms} demonstrated that accessibility holds for a $C^1$-open and dense subset of all partially hyperbolic diffeomorphisms, settling the $C^1$ case of the conjecture without restrictions on the center bundle's dimension.  

Accessibility has garnered considerable attention as a key geometric property with numerous critical applications in partially hyperbolic dynamics. As highlighted by the Pugh--Shub Stable Ergodicity Conjecture, accessibility serves as a powerful tool to establish ergodicity---a result proven for systems with one-dimensional center bundles \cite{08invent} and under center-bunched conditions \cite{BW10annals}. Building on Brin's work \cite{Brin75_transitive}, accessibility also provides an effective technique to verify transitivity. Furthermore, accessibility enables the generalization of periodic obstructions to solving the cohomological equation and Livsic regularity problems for partially hyperbolic diffeomorphisms \cite{Wilkinson_cohomological}. Additionally, it plays a significant role in analyzing the properties of special ergodic measures, as detailed in \cite{HHTU12mme, VY13, AVW-flow, CP}. We refer the reader to \cite{Wilkinson10-survey, 2018survey} for surveys and related discussion.

As formulated by Conjecture \ref{PS} and results from \cite{08invent, BHHTU08}, accessibility is a prevalent property among partially hyperbolic diffeomorphisms in dimension three, regardless of volume preservation. This prevalence raises the fundamental question of determining whether a given partially hyperbolic diffeomorphism is accessible. The central challenge involves characterizing the meager subset of non-accessible systems within this class.

Substantial efforts have been dedicated to verifying the accessibility property across canonical 3-manifold structures. Hertz, Hertz, and Ures \cite{2008nil} identified the first class of manifolds where all conservative partially hyperbolic systems are necessarily accessible. Building on the contributions of \cite{HamU14CCM}, Gan and Shi \cite{GS20DA} established a necessary and sufficient condition for accessibility on $\mathbb{T}^3$ for conservative partially hyperbolic diffeomorphisms homotopic to Anosov automorphisms. Subsequently, Hammerlindl and Shi \cite{HS21DA} extended these results to non-conservative systems, providing a characterization of accessibility classes in closed 3-manifolds with solvable fundamental groups.  

Recent work has shifted focus to manifolds with non-solvable fundamental groups. Under the non-wandering assumption, Hammerlindl--Hertz--Ures \cite{2020Seifert} and Fenley--Potrie \cite{FP_hyperbolic} demonstrated accessibility for Seifert manifolds in certain isotopy classes; Fenley and Potrie \cite{FP_hyperbolic} further established accessibility in hyperbolic 3-manifolds.  

An alternative approach to address accessibility involves adopting a dynamical perspective rather than focusing solely on manifold topology. A prominent class of partially hyperbolic systems consists of time-one maps of Anosov flows. For manifolds with non-solvable fundamental groups, Fenley and Potrie \cite{FP_hyperbolic} established the accessibility of discretized Anosov flows. Subsequently, they \cite{FP21accessible} considered a broader class of partially hyperbolic diffeomorphisms—encompassing all discretized Anosov flows and anomalous examples constructed in \cite{BPP1,BGP2,BGHP3}, termed collapsed Anosov flows \cite{BFP23collapsed}—and proved accessibility when the manifold's fundamental group is not virtually solvable and the non-wandering set equals the entire manifold.  

The problem of determining accessible partially hyperbolic systems is also explored in our prior work. Under non-wandering condition and assuming a non-solvable fundamental group, we demonstrated accessibility for partially hyperbolic diffeomorphisms without periodic points \cite{FU1}, systems with quasi-isometric centers \cite{Feng24}, and those lying in the homotopy class of the identity \cite{FU-id}. Crucially, these results involve no manifold-specific constraints beyond the non-solvable fundamental group requirement.  

As suggested in \cite{DW03,BHHTU08,HS21DA,FP_hyperbolic}, accessibility may be provable without assuming non-wandering behavior or conservativity. This paper generalizes the result from \cite{FU1} to systems where the non-wandering set need not coincide with the manifold.  

\begin{thm}\label{accessible}  
    Let $f: M \to M$ be a partially hyperbolic diffeomorphism of a closed 3-manifold $M$ without periodic points. If $\pi_1(M)$ is not virtually solvable, then $f$ is accessible.  
\end{thm}  

The lack of a non-wandering assumption represents a significant advancement in accessibility research. Specifically, the class of diffeomorphisms studied here is broader than non-wandering or conservative systems. Arbitrarily near any product of an Anosov diffeomorphism and a circle rotation, there exists a $C^1$-open set of accessible yet non-transitive partially hyperbolic diffeomorphisms \cite{NT01}, which therefore have wandering points due to Brin's result \cite{Brin75_transitive}. Additionally, accessible discretized Anosov flows with proper attractors \cite{BonattiGuelman} and anomalous partially hyperbolic diffeomorphisms with chain-recurrent sets distinct from the manifold \cite{BPP1} provide further examples of systems with wandering points. Notably, we can modify the latter examples to eliminate periodic points, thereby distinguishing our result from the non-wandering case in \cite{FU1}.  

Combining Theorem \ref{accessible} with results from \cite{Ham17CMH,HS21DA}, we obtain an exclusive dichotomy on accessibility classes for all closed 3-manifolds in the absence of periodic points.  

\begin{cor}\label{accessible=nosutorus}  
    Let $f: M \to M$ be a $C^1$ partially hyperbolic diffeomorphism of a closed 3-manifold without periodic points. Then $f$ is accessible if and only if no $su$-tori exist. Furthermore, if $f$ is $C^2$ and possesses an $su$-torus, then $f$ admits an $su$-foliation by tori.  
\end{cor}  

We emphasize that the class of non-accessible partially hyperbolic diffeomorphisms remains non-empty even in the absence of periodic points. Such diffeomorphisms admit 2-tori tangent to $E^s \oplus E^u$. Due to the lack of periodic points, these tori are abundant, and none are periodic under the diffeomorphism. A canonical example is the product of an Anosov diffeomorphism on $\mathbb{T}^2$ and an irrational circle rotation, which yields a non-accessible partially hyperbolic system without periodic points. When $f$ is only $C^1$, the collection of $su$-tori may form a proper lamination rather than a foliation. Examples of this phenomenon can be constructed by inserting open accessibility classes within the complement of $su$-tori, analogous to Denjoy’s circle example.  

The existence of stable and unstable foliations is fundamental to Anosov theory. Foundational results in \cite{BrinPesin74, HPS77} established the unique integrability of the strong bundles $E^s$ and $E^u$ for partially hyperbolic systems (see Theorem \ref{HPS}). However, the center bundle $E^c$ is not always integrable, and invariant foliations tangent to $E^c \oplus E^s$ or $E^c \oplus E^u$ may fail to exist. A partially hyperbolic diffeomorphism is \emph{dynamically coherent} if both $E^c \oplus E^s$ and $E^c \oplus E^u$ admit tangent invariant foliations. In such cases, the intersection of these foliations yields a center foliation tangent to $E^c$.  

A canonical example of dynamical incoherence by Borel, originally cited by Smale \cite{Smale67} and later identified by Wilkinson \cite{Wilkinson98}, is a non-toral Anosov automorphism on a six-dimensional nilmanifold; see also \cite{BurnsWilkinson08}. It was previously believed that integrability of the center bundle would always hold for systems with one-dimensional center bundles, as the Frobenius condition is automatically satisfied. This belief persisted until Hertz, Hertz, and Ures \cite{2016example} constructed a non-integrable center bundle example on $\mathbb{T}^3$—the first instance of central non-integrability. This result revealed another obstruction to center bundle unique integrability: the lack of differentiability in the one-dimensional center bundle. While Peano's theorem guarantees the existence of curves tangent to a non-smooth one-dimensional center bundle, these curves may not be assembled as a foliation. Notably, the Hertz--Hertz--Ures example is pointwise partially hyperbolic but not absolutely partially hyperbolic, as demonstrated by Brin--Burago--Ivanov \cite{BBI09}. Absolute partial hyperbolicity, a stricter condition requiring the uniform domination constants independent of the base point, was later shown by Bonatti--Gogolev--Hammerlindl--Potrie \cite{BGHP3} to coexist with robust center non-integrability.  

The central role of center bundle integrability in partially hyperbolic dynamics motivates a longstanding open question:  

\begin{qus}  
    Under what conditions are partially hyperbolic systems dynamically coherent?  
\end{qus}  

In this paper, we establish dynamical coherence for non-accessible partially hyperbolic diffeomorphisms with one-dimensional center bundles, no periodic points, and acting on closed manifolds of dimension $n \geq 3$.  

\begin{thm}\label{integrability}  
    Let $f: M \to M$ be a partially hyperbolic diffeomorphism on a closed manifold with $\dim E^c = 1$. If $f$ has no periodic points and is not accessible, then $E^c$ is uniquely integrable. Consequently, $f$ is dynamically coherent.  
\end{thm}  

We provide a complete depiction of accessibility classes for partially hyperbolic diffeomorphisms on closed manifolds of dimension $n \geq 3$. Let $\Gamma(f)$ denote the set of non-open accessibility classes (see Section~\ref{PH} for details).  

\begin{thm}\label{codim1}  
    Let $f: M \to M$ be a partially hyperbolic diffeomorphism on a closed manifold with $\dim E^c = 1$ and no periodic points. Then exactly one of the following holds:  
    \begin{enumerate}  
        \item $f$ is accessible;  
        \item $\Gamma(f)$ contains infinitely many compact accessibility classes;  
        \item $\Gamma(f)$ forms an invariant minimal $su$-foliation with non-compact leaves;  
        \item $\Gamma(f)$ contains a unique invariant minimal sub-lamination (possibly equal to $\Gamma(f)$) that extends trivially to a (not necessarily invariant) foliation by non-compact leaves.  
    \end{enumerate}  
    Furthermore, if $NW(f) = M$, the fourth case cannot occur.  
\end{thm}  

This result generalizes the three-dimensional case in \cite{2008nil} to higher dimensions using the lack of periodic points rather than the non-wandering assumption, where compact accessibility classes in the three-dimensional setting are incompressible tori.

Classical examples of three-dimensional partially hyperbolic diffeomorphisms include systems derived from Anosov dynamics, skew products over Anosov automorphisms of $\mathbb{T}^2$, and time-varying maps of Anosov flows. The classification problem for such systems originates from an informal conjecture by Pujals, later formalized by Bonatti and Wilkinson \cite{BW05}, which posits that these three classes constitute the foundational building blocks for all transitive partially hyperbolic diffeomorphisms on $3$-manifolds. Successful classifications have been achieved for specific classes of $3$-manifolds \cite{HP14,HP15,BFFP1,BFFP2}, under rigid geometric constraints \cite{CPH21}, and under certain dynamical restrictions \cite{BZ20,Feng24,EMP24}. However, the discovery of anomalous examples \cite{BPP1,BGP2,BGHP3} has introduced significant complexity to this classification program.  

Recent advances \cite{BFP23collapsed,FP23_transverse,FP-gafa} address these challenges by expanding the classification framework through the lens of \emph{collapsed Anosov flows}, a concept unifying classical and anomalous examples. This approach provides new insights into the structural taxonomy of partially hyperbolic systems. For comprehensive overviews of this evolving field, we direct readers to the surveys \cite{2018survey,HP18_survey,Potrie18ICM}.  

Notably, most existing classification results for 3-dimensional systems require either the non-wandering property (or transitivity) or dynamical coherence. In contrast, we provide a classification free from both non-wandering and dynamical coherence assumptions, fully characterizing non-accessible partially hyperbolic diffeomorphisms without periodic points.  

\begin{thm}\label{classification}  
    Let $f: M \to M$ be a partially hyperbolic diffeomorphism of a closed 3-manifold without periodic points. Then either:  
    \begin{itemize}  
        \item $f$ is accessible; or  
        \item after a finite lift and iteration, $f$ belongs to one of the following classes:  
        \begin{itemize}  
            \item Leaf conjugacy to a skew product over a linear Anosov automorphism of $\mathbb{T}^2$;  
            \item Discretized suspension Anosov flow.  
        \end{itemize}  
    \end{itemize}  
\end{thm}  

This demonstrates that any partially hyperbolic diffeomorphism avoiding the two model classes above must be accessible in the absence of periodic points.  

The structure of this paper is as follows. In Section \ref{Preliminaries}, we review foundational concepts and prior useful results for our analysis. Section \ref{section3} focuses on systems with one-dimensional center bundles, where we present the proofs of Theorem \ref{integrability} and Theorem \ref{codim1}. In Section \ref{section4}, we first establish the completeness of the center-stable foliation for systems with one-dimensional center bundles (Proposition \ref{complete}). We then specialize to three-dimensional manifolds to prove Theorem \ref{accessible}. Finally, Section \ref{consequence} explores further implications, including the classification result (Theorem \ref{classification}) and plaque expansiveness properties (Theorem \ref{plaqueexpansivity}).

\section{Preliminaries}\label{Preliminaries}

\subsection{Partial hyperbolicity}\label{PH}

Let $M$ be a compact Riemannian manifold. We say a diffeomorphism $f: M\rightarrow M$ is \emph{partially hyperbolic} if the tangent bundle of $M$ splits into three nontrivial invariant subbundles $TM=E^s \oplus E^c \oplus E^u$ such that for an adapted metric and for all $x\in M$ and unit vectors $v^\sigma \in E^\sigma_x$ ($\sigma= s, c, u$), 
\begin{equation*}
	\|Df(x)v^s\|<1<\|Df(x)v^u\| \\
	\qquad and \qquad
	\|Df(x)v^s\|< \|Df(x)v^c\|< \|Df(x)v^u\|.
\end{equation*}

A canonical class of partially hyperbolic diffeomorphisms comprises Anosov systems exhibiting dominated splitting into at least three invariant subbundles, where the center bundle retains hyperbolicity. As a natural generalization, partially hyperbolic systems share fundamental dynamical properties with Anosov systems, particularly regarding their strong bundles. For instance, strong bundles always satisfy unique integrability.  

\begin{thm}\cite{BrinPesin74}\cite{HPS77}\label{HPS}  
    Let $f: M \to M$ be a partially hyperbolic diffeomorphism. Then there exist unique invariant foliations $\mathcal{F}^s$ and $\mathcal{F}^u$ tangent to $E^s$ and $E^u$ respectively.  
\end{thm}  

The foliations $\mathcal{F}^s$ and $\mathcal{F}^u$ are called the \emph{stable foliation} and \emph{unstable foliation} respectively. However, the center bundle $E^c$ is not necessarily integrable, and its integral manifolds may lack uniqueness even when a center foliation exists. The integrability of $E^c$ remains a longstanding challenge.  

For one-dimensional center distributions, the existence theorem of ordinary differential equations guarantees a $C^1$ curve through every point tangent to $E^c$. The local stable manifold of such a curve forms a $C^1$ disk tangent to $E^s \oplus E^c$ with dimension $\dim E^s + \dim E^c$. Consequently, through each point passes an immersed complete $C^1$ manifold everywhere tangent to $E^s \oplus E^c$. However, these integral manifolds may self-intersect or fail to form a foliation \cite{BurnsWilkinson08, BBI09}.  

Crucially, the existence of immersed integral manifolds is non-trivial for arbitrary joint distributions and depends on the dynamics of complementary invariant subbundles. The joint bundle $E^s \oplus E^u$ exhibits fundamentally distinct behavior, as we shall demonstrate.  

A partially hyperbolic diffeomorphism is \emph{dynamically coherent} if there exist invariant foliations $\mathcal{F}^{cs}$ and $\mathcal{F}^{cu}$ tangent to $E^s \oplus E^c$ and $E^c \oplus E^u$ respectively. We call $\mathcal{F}^{cs}$ (resp. $\mathcal{F}^{cu}$) the \emph{center-stable} (resp. \emph{center-unstable}) foliation. Under dynamical coherence, intersecting $\mathcal{F}^{cs}$ and $\mathcal{F}^{cu}$ yields an $f$-invariant center foliation $\mathcal{F}^c$ tangent to $E^c$. When $E^c$ is integrable, the invariant foliation $\mathcal{F}^c$ tangent to $E^c$ is called the \emph{center foliation}. The distribution $E^c$ is \emph{uniquely integrable} if there exists an $f$-invariant $C^0$ foliation $\mathcal{F}^c$ with $C^1$ leaves such that every embedded $C^1$ curve $\sigma: [0,1] \to M$ satisfying $\dot{\sigma}(t) \in E^c(\sigma(t))$ for all $t \in [0,1]$ lies entirely within the leaf $\mathcal{F}^c(\sigma(0))$.  

Unique integrability of $E^c$ implies the existence of center foliation, center-stable and center-unstable foliations. In particular, the center unique integrability implies dynamical coherence. However, dynamical coherence does not guarantee unique integrability. A center-stable manifold may exist without being a leaf of any center-stable foliation (see \cite{2016example,2018survey}). Thus unique integrability is strictly stronger than dynamical coherence despite both implying the existence of $\mathcal{F}^c$. For variations on center integrability, see \cite{BurnsWilkinson08}.  

In \cite{2015Center}, Hertz, Hertz, and Ures established:  

\begin{thm}\label{nocompactleaf}\cite{2015Center}  
    For any dynamically coherent partially hyperbolic diffeomorphism $f: M^3 \to M^3$ on a closed 3-manifold, the center-stable and center-unstable foliations admit no compact leaves.  
\end{thm}  

Furthermore, 3-manifolds containing compact tori tangent to $E^s \oplus E^c$, $E^c \oplus E^u$, or $E^s \oplus E^u$ are fully classified. The following theorem demonstrates that such $cs$-, $cu$-, or $su$-tori imply virtual solvability of the fundamental group:  

\begin{thm}\label{maptori}\cite{2011TORI}  
    Let $f: M^3 \to M^3$ be a partially hyperbolic diffeomorphism on a closed orientable 3-manifold. If there exists an $f$-invariant embedded 2-torus $T$ tangent to $E^s \oplus E^u$, $E^c \oplus E^u$, or $E^c \oplus E^s$, then $M^3$ must be one of:  
    \begin{enumerate}  
        \item the 3-torus $\mathbb{T}^3$;  
        \item the mapping torus of $-id: \mathbb{T}^2 \to \mathbb{T}^2$;  
        \item the mapping torus of a hyperbolic automorphism on $\mathbb{T}^2$.  
    \end{enumerate}  
\end{thm}  

Dynamically coherent partially hyperbolic diffeomorphisms on these three manifold types are classified up to leaf conjugacy:  

\begin{thm}\cite{Hammerlindl13,HP14,HP15}\label{classification_sol}  
    Let $f: M \to M$ be a dynamically coherent partially hyperbolic diffeomorphism on a closed 3-manifold with virtually solvable fundamental group. Then, after finite lift and iteration, $f$ is leaf-conjugate to one of:  
    \begin{enumerate}  
        \item a linear Anosov diffeomorphism on $\mathbb{T}^3$;  
        \item a skew product over a linear Anosov diffeomorphism of $\mathbb{T}^2$;  
        \item the time-one map of a suspension Anosov flow.  
    \end{enumerate}  
\end{thm}  

A set is \emph{$s$-saturated} (resp. \emph{$u$-saturated}) if it is a union of stable (resp. unstable) leaves. A set is \emph{$su$-saturated} if it is both $s$- and $u$-saturated. For each $x \in M$, the \emph{accessibility class} $AC(x)$ is the minimal $su$-saturated set containing $x$—equivalently, any two points in $AC(x)$ are connected by an $su$-path (a piecewise $C^1$ curve tangent to $E^s \cup E^u$). Accessibility classes partition $M$. The diffeomorphism $f$ is \emph{accessible} if $M$ is the only accessibility class, which occurs precisely when $\Gamma(f) = \emptyset$, where $\Gamma(f)$ denotes the set of non-open accessibility classes.  

The bundles $E^s$ and $E^u$ are \emph{jointly integrable} at $x \in M$ if there exists $\delta > 0$ such that for all $y \in W^s_\delta(x)$ and $z \in W^u_\delta(x)$, the local manifolds satisfy $W^u_{\mathrm{loc}}(y) \cap W^s_{\mathrm{loc}}(z) \neq \emptyset$.  

\begin{prop}\label{jointint}\cite{08invent}  
    Let $f: M \to M$ be a partially hyperbolic diffeomorphism with $\dim E^c = 1$. Then $\Gamma(f)$ is a compact codimension-one invariant set laminated by accessibility classes. A point $x$ lies in $\Gamma(f)$ if and only if $E^s$ and $E^u$ are jointly integrable at every point in $AC(x)$. Moreover, $E^s$ and $E^u$ are jointly integrable everywhere if and only if $\Gamma(f) = M$.  
\end{prop}  

The following proposition provides a sufficient condition for unique center integrability:  

\begin{prop}\label{Gamma(f)=M}\cite{2006Some}  
    Let $f: M \to M$ be a partially hyperbolic diffeomorphism with $\dim E^c = 1$. If $f$ has no periodic points and $E^s$, $E^u$ are jointly integrable, then $E^c$ is uniquely integrable.  
\end{prop}  

We emphasize that Proposition \ref{Gamma(f)=M} is a direct consequence of Theorem \ref{integrability}, but not equivalent: $\Gamma(f)$ may form a proper sublamination rather than a full foliation.

\subsection{Laminations}

A \emph{lamination }$\mathcal{L}$ of a compact Riemannian manifold $M$ is defined by a smooth atlas with charts of the form $\mathbb{R}^n\times T^k$, where $T^k$ is a locally compact subset of $\mathbb{R}^k$. More precisely, the collection of open charts $\{U_i\}$ with local product coordinate maps $\phi _i: U_i\rightarrow \mathbb{R}^n\times T^k$ covers $\mathcal{L}$ and the transition maps $\phi_j \circ \phi^{-1}_i: \mathbb{R}^n\times T^k\rightarrow \mathbb{R}^n\times T^k$ are homeomorphisms acting as 
$$\phi_j \circ \phi^{-1}_i(r,t)=(\psi_{ij}(r,t), \eta_{ij}(t) ),$$
where $\psi_{ij}$ are $C^1$ for the first coordinate $r\in \mathbb{R}^n$. The transverse direction $T^k$ need not admit differentiability. Here we are concerned with codimension-one laminations, which means $k=1$. The sets $\phi^{-1}_{i}(\mathbb{R}^n\times \{t\})$ are called \emph{plaques} of the lamination $\mathcal{L}$ and the sets $\phi^{-1}_{i}(\{r\}\times T^k)$ (or $\phi^{-1}_{i}(\{r\}\times \mathbb{R}^k)$) are \emph{transversals} to $\mathcal{L}$. Through any point $x\in M$, there are some plaques in each $U_i$ connected to a maximal injectively immersed $n$-manifold--the \emph{leaf} of $\mathcal{L}$ through $x$. Each pair of such immersed manifolds is either identical or disjoint. Given a chart $(U, \phi)$ on $M$, an open subset $U$ is a \emph{coordinate cube} if $\phi(U)$ is an open cube of $\mathbb{R}^n\times \mathbb{R}^k$.

A \emph{foliation} $\mathcal{F}$ is a lamination of $M$ that decomposes $M$ into leaves. Notice that the leaves of a foliation form a partition of $M$, and a foliation is a lamination with an empty complement. While foliation leaves are $C^1$ and possibly smooth, the foliation itself may only be continuous. This paper focuses primarily on codimension-one laminations (one-dimensional transversals).  

For general laminations or foliations, the collection of compact leaves may not form a compact set. However, this holds for codimension-one laminations:  

\begin{thm}\label{Hae62}\cite{Hae62}  
    The set of compact leaves of a $C^0$ codimension-one lamination forms a compact sublamination.  
\end{thm}

We remark that this property was originally stated for foliations and, in fact, it holds for laminations, see for instance \cite{HectorHirsch}.

Consider a codimension-one proper lamination $\mathcal{L}$ that is not a foliation on $M$. A \emph{complementary region} $V$ of $\mathcal{L}$ is a path connected component in $M\setminus \mathcal{L}$. A \emph{closed complementary region} $\hat{V}$, which may be not compact, is a completion of a complementary region $V$ with respect to the path metric induced by the Riemannian metric. A leaf $L$ is called a \emph{boundary leaf} of the lamination $\mathcal{L}$ if for every point $x\in L$, there exists an arc $\alpha$ through $x$ satisfying $\alpha\cap \mathcal{L}=\{x\}$. We also say that $L$ is a boundary leaf of a complementary region $V$ if the arc $\alpha: [0,1]\rightarrow M$ above can be chosen with $\alpha(0)=x\in L$ and $\alpha\left(0,1\right]\subset V$. The union of boundary leaves of $V$ is denoted by $\partial \hat{V}$. When $\mathcal{L}$ has no isolated leaves, the set $\hat{V}$ is exactly the union of $V$ and its boundary leaves.

The preceding concepts apply to arbitrary dimensions. Let us introduce some terminologies specializing to three-manifolds. A surface $S$ is \emph{incompressible} if the inclusion map $i_*:=\pi_1(S)\hookrightarrow\pi_1(M)$ is injective; or equivalently, for any embedded disk $D$ with $D\cap S=\partial D$, the boundary $\partial D$ is null-homotopic in $S$. Let $S$ be a properly embedded surface in $M$, that is, its boundary $\partial S$ is completely contained in the boundary $\partial M$ of $M$. We say that $S$ is \emph{end-incompressible} if for each disk $D\subset M$ whose boundary $\partial D$ consists of a curve $\beta:=D\cap S$ and a boundary curve $\gamma:= D\cap \partial M$, there is a disk $D'\subset S$ such that $\partial D'$ consists of the curve $\beta$ and a boundary $\partial D'\cap \partial M\subset \partial S$. In a boundaryless manifold, an incompressible surface is also end-incompressible. A 3-manifold is \emph{irreducible} if every embedded 2-sphere in the manifold bounds a 3-ball. Similarly, we say that a subset $A\subset M$ is \emph{irreducible} if every embedded 2-sphere in $A$ bounds a 3-ball in $A$.
A lamination is \emph{essential} if it contains no spherical leaf or torus leaf bounding a solid torus, and furthermore any closed complementary region $\hat{V}$ is irreducible and the boundary leaves of $\partial \hat{V}$ are both incompressible and end-incompressible in $\hat{V}$. 

For closed 3-manifolds not homeomorphic to $\mathbb{T}^3$, every essential lamination must contain a non-planar leaf, as established by:  

\begin{thm}\cite{Rosenberg,Gabai90}\label{rosenberg}  
    If a closed 3-manifold admits a $C^0$ essential lamination consisting entirely of planes, then it must be the 3-torus.  
\end{thm}  

Rosenberg \cite{Rosenberg} originally proved this for $C^2$ foliations; Gabai \cite{Gabai90} extended it to essential laminations using \cite[Theorem 3.1]{Imanishi74}. Consequently, the theorem holds for $C^0$ foliations.  

An \emph{$I$-bundle} is a fiber bundle with interval fibers—locally homeomorphic to a product of a base set and an interval. When a lamination $\mathcal{L}$ has no compact leaves, each closed complementary region decomposes uniquely (up to isotopy) along annuli into a compact \emph{gut} and finitely many \emph{interstitial regions}. These interstitial regions are $I$-bundles over noncompact subsets of boundary leaves and can be made arbitrarily thin away from the gut, see for instance \cite{HectorHirsch,CC00I,Calegari07book}.  


\section{Central integrability}\label{section3}

In this section, we consider a closed $n$-dimensional Riemannian manifold $M^n$ ($n \geq 3$) and a partially hyperbolic diffeomorphism $f: M^n \to M^n$ with one-dimensional center bundle $E^c$. We will study the integrability of $E^c$ and characterize all accessibility classes under the absence of periodic points.  

We assume $M^n$ and $E^c$ are orientable with orientations preserved by $f$. This is achievable by passing to a finite cover and iterate if necessary. There is no loss of generality since unique integrability is local, and accessibility persists under finite lifts and iterates. This section primarily establishes the proof of Theorem \ref{integrability}.  

Throughout, we fix an orientation on $E^c$. When $E^c$ is not uniquely integrable at a point, ``center curves" refers to distinct integral curves sharing this orientation unless otherwise specified.  

Let us introduce a lemma followed by continuity and transversality of the invariant bundles $E^s$, $E^c$, and $E^u$.

\begin{lem}\label{lem1}\cite[Lemma 3.6]{2006Some}
	For $\epsilon >0$ there exists $\delta >0$ such that if $d(x,y)<\delta$ and $z\in W^{c}_{\delta}(x)$, then $W^{c}_{loc}(y) \cap W^{s}_{\epsilon}(W^{u}_{\epsilon}(z)) \neq \emptyset$, regardless of the choice of center curves through x and y. Moreover, for any small $\delta>0$, there exists $\rho>0$ such that if $d(x, y)<\rho$, then we have $W^c_{\delta}(x)\cap W^s_{loc}(W^u_{loc}(y))\neq \emptyset$ for any choice of center curves through $x$.
\end{lem}

When the center bundle $E^c$ is not uniquely integrable at some point $x$, there exist two distinct center arcs $c_1$ and $c_2$ through $x$. One can see that there are two distinct center arcs lying in either the same local center-stable manifold or the same local center-unstable manifold through $x$. Indeed, suppose that $c_1$ and $c_2$ belonged to different local center-stable and center-unstable manifolds. Then, by transversality, the local stable manifold of $c_1$ intersects the local unstable manifold of $c_2$, yielding a third distinct center arc through $x$ sharing center-stable manifold with $c_1$ and center-unstable manifold with $c_2$.

In the subsequent, when $E^c$ is not uniquely integrable at $x$, we will always assume that both center arcs through $x$ lie in either the same center-stable or the same center-unstable manifold. More precisely, when $E^s\oplus E^c$ (resp. $E^c\oplus E^u$) is not uniquely integrable at $x$, we assume that both center arcs through $x$ lie in the same center-unstable (resp. center-stable) manifold. We will show the unique integrability of $E^c$ through proving unique integrability of $E^s\oplus E^c$ and $E^c\oplus E^u$.

With this observation, one have the following lemma:

\begin{lem}\label{lem2}\cite[Remark 3.7]{2006Some}
	If $E^s\oplus E^{c}$ is non-uniquely integrable at x, then for sufficiently small $\delta >0$, and for each connected center subsegment containing x, say c, in one of the two separatrices, there is $N>0$ for which $f^{n}(c) \not\subset B_{\delta}(f^{n}(x))$ for all $n\geq N$. Analogous statement holds for $f^{-1}$ when $E^c\oplus E^u$ is not uniquely integrable.
\end{lem}

Our strategy in providing unique integrability of the center bundle under non-accessible condition is analyze points across different accessibility classes. 
As established in Proposition \ref{jointint}, the set $\Gamma(f)$ forms a compact invariant lamination of codimension one, consisting precisely of non-open accessibility classes. We will analyze the case where $\Gamma(f)$ is non-empty and strictly contained in $M$, treating it as a proper invariant lamination of codimension one.

\subsection{Lamination with a compact leaf}

We require the following variant of the Anosov closing lemma; see for instance \cite[Lemma 3.8]{Bowen_book} for a proof.  

\begin{lem}(Anosov Closing Lemma)\label{closinglemma}  
    There exists $\epsilon_0 > 0$ such that for any $x \in \Gamma(f)$, if $f^k\left(W^s_{\epsilon_0}(W^u_{\epsilon_0}(x))\right) \cap W^s_{\epsilon_0}(W^u_{\epsilon_0}(x)) \neq \emptyset$ for some $k > 0$, then $W^s_{\epsilon_0}(W^u_{\epsilon_0}(x))$ contains a periodic point.  
\end{lem}  

Given an orientation on $E^c$, denote by $E^c_+(x)$ and $E^c_-(x)$ the positively and negatively oriented half-lines at $x \in M$, respectively. Let $\mathcal{C}^+(x)$ and $\mathcal{C}^-(x)$ be the collections of $C^1$ integral curves of $E^c_+(x)$ and $E^c_-(x)$ passing through $x$ as a common endpoint.  

\begin{defn}  
    The bundle $E^c_+$ (resp. $E^c_-$) is \emph{uniquely integrable at $x$} if there exists $\alpha \in \mathcal{C}^+(x)$ (resp. $\mathcal{C}^-(x)$) such that every sufficiently short embedded $C^1$ curve $\sigma: [0,1] \to M$ satisfying $\dot{\sigma}(t) \in E^c_+(\sigma(t))$ (resp. $E^c_-(\sigma(t))$) and $\sigma(0) = x$ lies entirely within $\alpha$.  
\end{defn}

One can see that $E^c$ is uniquely integrable at a point if and only if both $E^c_+$ and $E^c_-$ are uniquely integrable there.

\begin{lem}\label{integ-accumulated}  
    Let $x \in M$ be an accumulation point of a sequence of plaques $\{P_i\} \subset \Gamma(f)$. If every sufficiently short curve in $\mathcal{C}^+(x)$ intersects infinitely many $P_i$, then $E^c_+$ is uniquely integrable at $x$. The analogous conclusion holds for $E^c_-$.  
\end{lem}  
\begin{proof}
    Obviously, the point $x$ is contained in $\Gamma(f)$ due to the compactness of $\Gamma(f)$ (see Proposition~\ref{jointint}). Suppose that $E^c_+$ is not uniquely integrable at $x$. Take a point $z\in \omega(x)$, then the iterations of $z$ under $f$ are all distinct points in $M$ since $f$ has no periodic point. By the compactness of $M$, there is a small open set $U$ of diameter $\delta/2$, for a small $\delta$ chosen as in Lemma \ref{lem1}, containing three iterates of $z$, denoted by $z_1, z_2,z_3$. Notice that each pair of iterations of $z$ cannot belong to the same $su$-plaque, otherwise we have a periodic point by the Anosov Closing Lemma (Lemma \ref{closinglemma}). Let $U'$ be a $\delta$-neighborhood of $U$. As $E^c$ is one-dimensional and oriented, the $su$-plaques of $\Gamma(f)$ can be ordered in $U'$. Here, we allow all $su$-plaques of $\Gamma(f)$ to project to a disconnected set in an open interval when $U'$ contains some points in the complement of $\Gamma(f)$. For this order, we assume that $z_1<z_3<z_2$. 
	
	As $z_i\in \omega(x)$ for $i=1,2, 3$, there are $n_1, n_2\in \mathbb{N}$ such that $d(f^{n_1}(x), z_1)\leq \delta/8$ and $d(f^{n_2}(x), z_2)\leq \delta/8$. Notice that $E^c_+$ is neither uniquely integrable at $f^{n_1}(x)$ or $f^{n_2}(x)$. Let $c_i\in \mathcal{C}^+(f^{n_i}(x))$ be a center curve through $f^{n_i}(x)$ of length less than $\delta/8$ for each $i=1, 2$. By assumption, the curve $c_i$ intersects $\Gamma(f)$ in infinitely many points. Up to taking a sub-curve, we assume that both endpoints of $c_i$ are contained in $\Gamma(f)$ for $i=1, 2$. By Lemma \ref{lem2}, there is an integer $N_i\in \mathbb{N}$ such that the length of $f^n(c_i)$ is greater than $\delta$ for any $n\geq N_i$, $i=1,2$. Pick a large number $n_3\in \mathbb{N}$ such that $n_3-n_i>N_i$ and $d(f^{n_3}(x), z_3)\leq \delta/8$ for $i=1,2$. Projecting along $su$-plaques, for either $i=1$ or $2$, we obtain that the map $f^{n_i-n_3}$ acts contractively on an interval $I$ that has one endpoint corresponding to the $su$-plaque through $f^{n_3}(x)$. By the $f$-invariance of $\Gamma(f)$, both endpoints of $I$ correspond to $su$-plaques in $\Gamma(f)$. Then, there is a fixed point of $f^{n_i-n_3}$ in $I$, which corresponds to an $su$-plaque in $\Gamma(f)$ since $\Gamma(f)$ is compact. This $su$-plaque contains a point $p$ and its iteration $f^{n_i-n_3}(p)$. Thus, by Lemma \ref{closinglemma}, this yields a periodic point of $f$, which is absurd.
\end{proof}

Now, let us first consider the case where there exists a compact accessibility class.

\begin{prop}\label{integ_compactleaf}
	Assume that $f$ has no periodic points and $\emptyset \neq \Gamma(f) \neq M$. If the lamination $\Gamma(f)$ contains a compact leaf, then $E^c$ is uniquely integrable.
\end{prop}

\begin{proof}
	Under the existence of a compact leaf of $\Gamma(f)$, the set $\Lambda$ of all compact leaves forms a compact invariant lamination by Theorem \ref{Hae62}. Suppose that there are finitely many compact leaves in $\Lambda$. Then we can find a periodic one of some period $K\in \mathbb{N}$. The restriction of $f^K$ on this periodic leaf is an Anosov diffeomorphism. It turns out the existence of periodic points of $f$ and thus contradicts the assumption. 
	Then, there are infinitely many compact leaves in $\Lambda$ and no one is periodic. 
	
	Suppose that $E^c_+$ is not uniquely integrable at some point $x$. Then, $\mathcal{C}^+(x)$ consists of distinct center curves through $x$ as the only common point. Let $\delta>0$ and $\rho>0$ be two sufficiently small constants given in Lemma \ref{lem1} and Lemma \ref{lem2}. By Lemma~\ref{lem2}, there is $N\in \mathbb{N}$ such that the length of $f^n(c)$ is larger than $\delta$ for any center curve $c\in \mathcal{C}^+(x)$ and any $n\geq N$ or $n\leq -N$. We assume the case $n\geq N$ and an analogous argument applies directly for the other. Suppose that there exists a curve $c\in \mathcal{C}^+(x)$ entirely contained in a connected compact region $R\subset\Lambda$. Notice that $\Lambda$ is a compact sublamination. Denote by $\hat{R}\subset \Lambda$ the maximal connected region containing $R$, which has several compact boundary leaves. As no leaf in $\Lambda$ is periodic, the regions $f^{i}(\hat{R})$ and $f^{j}(\hat{R})$ are disjoint for any $i\neq j\in \mathbb{Z}$. However, for each $n\geq N$, $f^n(\hat{R})$ contains a center curve $f^n(c)$ of length larger than $\delta$. Then, by Lemma \ref{lem1}, $f^n(\hat{R})$ contains a ball of radius bounded from below by $\rho$ for every $n\geq N$. This contradicts the compactness of the manifold $M$. 
	
	Suppose that there exists a center curve $c\in\mathcal{C}^+(x)$ entirely contained in the complement of $\Lambda$ or intersecting $\Lambda$ in the unique point $x$ (which happens if $x\in \Lambda$). Denote by $V$ the maximal connected component of the complement $M\setminus\Lambda$ containing $c$, and $L\in \Lambda$ one of its boundary leaves. As $\Lambda$ is compact, it contains a compact leaf $L_0\in \Lambda$ accumulated by $f^{n_i}(L)$. There is an integer $m\in \mathbb{N}$ such that $f^{n_i}(L)$ is contained in the $\rho$-neighborhood $U_{\rho}$ of $L_0$ for each $i\geq m$. It implies that $f^{n_i}(V)$ belongs to $U_\rho$ for each $i\geq m$. Using Lemma~\ref{lem1} and Lemma~\ref{lem2}, there are integers $j\neq i\geq m$ with $n_j\neq n_i\geq N$ such that $f^{n_i}(c)$ intersects both $f^{n_i}(L)$ and $f^{n_j}(L)$, which are two distinct leaves in $\Lambda$. By the $f$-invariance of $\Lambda$, we obtain that $c$ intersects $\Lambda$ in two distinct points, arriving at a contradiction.

    Now, the only remaining case is that $x$ is accumulated by a sequence of leaves $L_i\in\Lambda$ and any curve in $\mathcal{C}^+(x)$ intersects both the sequence $L_i$ and $M\setminus\Lambda$. In particular, one can deduce that $x\in \Lambda$ by the compactness of $\Lambda$. Using Lemma~\ref{integ-accumulated}, we obtain the unique integrability of $E^c_+$ at $x$. The same argument applies analogously for the unique integrability of $E^c_-$ at $x$. Therefore, we obtain that $E^c$ is uniquely integrable at any point $x$ and thus finish the proof.
\end{proof}

In the three-dimensional case, any compact leaf of $\Gamma(f)$ must be a $2$-torus. This follows because leaves of $\Gamma(f)$ are foliated by the one-dimensional unstable foliation, and the torus is the only compact orientable surface admitting such a foliation. When $\Gamma(f)$ contains a compact accessibility class, Theorem \ref{maptori} implies that the orientable closed 3-manifold $M$ must be one of the following:  
(i) the 3-torus $\mathbb{T}^3$;  
(ii) the mapping torus of $-id: \mathbb{T}^2 \to \mathbb{T}^2$;  
(iii) the mapping torus of a hyperbolic automorphism of $\mathbb{T}^2$.  
Equivalently, $M$ is a mapping torus of a torus automorphism commuting with some hyperbolic automorphism.

\subsection{Integrability of the boundary leaves}\label{section_integ_boundary}

From now on, we assume that $\Gamma(f)$ is a non-empty compact proper lamination consisting entirely of non-compact leaves. As established in Section~\ref{Preliminaries}, each closed complementary region decomposes uniquely (up to isotopy) along annuli into a compact gut region and several connected non-compact interstitial regions. Each interstitial region is an $I$-bundle over a boundary leaf and can be made arbitrarily thin away from the gut.  

The following lemma is crucial for our proof:  

\begin{lem}\label{visit_inters_infty}  
    For any $x \in M$, if both the forward orbit $\{f^n(x)\}_{n\geq0}$ and backward orbit $\{f^n(x)\}_{n\leq0}$ intersect the interstitial regions infinitely often, then $E^c$ is uniquely integrable at $x$.  
\end{lem}

\begin{proof}
Choose interstitial regions sufficiently thin such that any center curve within them connects two boundary leaves, and every center segment between boundary leaves has length at most $\delta/2$, where $\delta > 0$ is the constant from Lemma \ref{lem2}. This is achievable because $E^c$ is transverse to $E^s \oplus E^u$, and interstitial regions can be made arbitrarily thin.

Suppose $E^s\oplus E^c$ is not uniquely integrable at some point $x$, and for every integer $n \in \mathbb{N}$, there exists $k > n$ such that $f^k(x)$ lies in an interstitial region. Fix a center curve $\gamma(x)$ through $x$. By Lemma \ref{lem2}, there exists $N > 0$ such that the length of $f^k(\gamma(x))$ is greater than $\delta$ for all $k > N$. Without loss of generality, assume $f^i(x)$ resides in an interstitial region for all $i > N$.

For each $i > N$, let $\alpha(f^i(x))$ be a connected subsegment of $f^i(\gamma(x))$ containing $f^i(x)$ and entirely contained within its interstitial region. Since non-unique integrability persists at $f^i(x)$, Lemma \ref{lem2} yields $N_i' > 0$ such that the length of $f^{k'}(\alpha(f^i(x)))$ is greater than $\delta$ for all $k' > N_i'$. Select $j > N_i'$ such that $f^{i+j}(x)$ also lies in an interstitial region. By $f$-invariance of the lamination decomposition, $f^j(\alpha(f^i(x)))$ remains entirely within an interstitial region. However, this curve simultaneously satisfies that the length is greater than $\delta$ yet must be at most $\delta/2$ by our thinness choice - a contradiction.

Thus, we conclude that $E^s\oplus E^c$ is uniquely integrable at $x$. The same argument for $f^{-1}$ implies the unique integrability of $E^c\oplus E^u$ at $x$. Hence, $E^c$ is uniquely integrable at $x$.
\end{proof}

Note that each closed complementary region of $\Gamma(f)$ containing a gut is periodic under $f$. Indeed, by the compactness of $M$, there are finitely many guts in the complement of $\Gamma(f)$. Any non-periodic complementary region would have its orbit intersecting interstitial regions infinitely often. By Lemma~\ref{visit_inters_infty}, this would imply $E^c$ is uniquely integrable in such regions. Furthermore, boundary leaves of periodic complementary regions are themselves periodic due to $f$-invariance of $\Gamma(f)$. Thus there exists a common period $k \in \mathbb{N}$ such that $f^k$ preserves all complementary regions containing guts, and their boundary leaves. Therefore, after passing to a finite iterate, we may assume every complementary region with a gut and its boundary leaves are $f$-invariant.

We define \emph{the boundary of a gut} as the set of points in boundary leaves accumulated by gut points. Similarly, \emph{the boundary of an interstitial region} comprises points in boundary leaves accumulated by points from the given interstitial region. Each boundary leaf of a complementary region consists precisely of the boundary of guts and the boundary of interstitial regions.

The following result is indicated by \cite[Chapter V, Section 3.2]{HectorHirsch} or \cite[Proposition 5.2.14]{CC00I}.

\begin{lem}\label{suplaque}  
    Every gut boundary contains only finitely many $su$-plaques.
\end{lem}

    

We now are going to establish unique integrability of the center bundle at boundary leaves of $\Gamma(f)$. The orientation of $E^c$ splits every center curve through boundary point $x$ into two connected arcs: one in $\mathcal{C}^+(x)$ and one in $\mathcal{C}^-(x)$. These represent two distinct directional possibilities from $x$. Crucially, unique integrability at boundary points requires demonstration in both directions simultaneously, since integrability might hold in one direction but fail in the other. This directional distinction is unnecessary for interior points of $\Gamma(f)$ or complementary regions, where both directions behave identically. However, boundary points demand separate consideration of each direction.

Consider a point $x$ in a boundary leaf of $\Gamma(f)$. We define a \emph{direction associated with complementary regions} as a direction containing some center curve through $x$ that lies entirely within a complementary region. Conversely, a \emph{direction associated with the lamination} contains only center curves that necessarily intersect $\Gamma(f)$ beyond $x$. Correspondingly, we describe center curves as either \emph{in the direction of complementary regions} or \emph{in the direction of the lamination}.

The following proposition establishes unique integrability of the center bundle at boundary leaves in directions associated with complementary regions. When both directions at a boundary point connect to complementary regions, identical reasoning yields unique integrability of $E^c$ in the complementary direction. Unique integrability in lamination-associated directions will be addressed in Proposition \ref{integ-rest}.

\begin{prop}\label{integ-boundary}
    The bundle $E^c$ is uniquely integrable at every boundary leaf point in the direction associated with complementary regions.
\end{prop}

\begin{proof}
    Consider first a point $x$ in the boundary of a gut. Its orbit contains at most finitely many points within that gut boundary. Suppose otherwise that infinitely many orbit points lie in the gut boundary. By Lemma~\ref{suplaque}, these occupy finitely many $su$-plaques. The pigeonhole principle implies one $su$-plaque contains two distinct orbit points. The Anosov Closing Lemma (Lemma \ref{closinglemma}) then produces a periodic point, contradicting our assumption.

    Now let $x$ be in the boundary of an interstitial region. The previous argument shows finitely many orbit points lie in any single gut boundary. Compactness of $M$ ensures finitely many guts, so altogether only finitely many orbit points occupy gut boundaries. Since boundary leaves partition into gut-boundary and interstitial-boundary components, infinitely many forward-orbit and backward-orbit points must lie in interstitial boundaries. Lemma \ref{visit_inters_infty} then establishes unique integrability in the complementary region direction.
\end{proof}

\subsection{Complementary regions and $I$-bundles}

This subsection focuses on complementary regions. We will establish unique integrability of $E^c$ at all interior points of $M \setminus \Gamma(f)$ by proving each complementary region forms an $I$-bundle over boundary leaves, foliated by uniformly bounded-length center segments. In fact, we provide a more general result showing the $I$-bundle structure for an arbitrary invariant compact (could be proper) lamination $\Lambda\subset\Gamma(f)$ without compact leaves.

By the lack of compact leaves, each closed complementary region of $\Lambda$ decomposes into a compact gut and finitely many interstitial regions. Denote by $V=M\setminus\Lambda$ the complement of $\Lambda$, $G$ the finite union of guts, and $I$ the collection of all interstitial regions. Let $\hat{V}$ be the closed complementary regions equipped with the path metric induced by the original metric. 

The following lemma first connects boundary leaves via center curves:

\begin{lem}\label{connect_two_leaves}
    For every point $x$ in a boundary leaf of $\Lambda\subset \Gamma(f)$, there exists a center curve connecting $x$ to another distinct boundary leaf.
\end{lem}

\begin{proof}
    It is immediate for $x$ in an interstitial region boundary. Since interstitial regions can be made arbitrarily thin, every center curve through $x$ terminates at the opposite boundary leaf.

    Consider a point $x$ in a gut boundary. The proof of Proposition~\ref{integ-boundary} implies only finitely many points in the orbit of $x$ lie in gut boundaries. Choose $N > 0$ such that $f^n(x)$ lies in an interstitial boundary for all $|n| \geq N$. By thinness of interstitial regions, every center curve through $f^n(x)$ connects two boundary leaves of a complementary region. Applying $f^{-n}$ preserves this connectivity (as the set of boundary leaves is $f$-invariant), yielding a center curve through $x$ connecting two boundary leaves.
\end{proof}

Observe that every complementary region of $\Lambda$ has at least two boundary leaves. While general laminations may admit complementary regions with more than two boundary leaves, our dynamical setting ensures exactly two:

\begin{cor}\label{two-boundary}
    Every closed complementary region of $\Lambda$ admits precisely two boundary leaves.
\end{cor}
\begin{proof}
    Let $U\subset V$ be a path-connected complementary region of $\Lambda$. Fix a boundary leaf $L_1$. By Lemma~\ref{connect_two_leaves}, every $x \in L_1$ has a center curve to another boundary leaf of $U$. Suppose $U$ has at least three boundary leaves. If all center curves from $L_1$ terminate at a single other boundary leaf $L'$, replace $L_1$ with $L'$ and repeat. Since $U$ has finitely many boundary leaves, we eventually find a boundary leaf (still called $L_1$) containing points $x,y$ whose center curves reach two distinct boundary leaves $L_2$ and $L_3$ respectively.
    
    Let $\{L_i\}_{i=1}^k$ be all boundary leaves of $U$. Define $A_i \subset L_1$ ($i=2,\dots,k$) as the set of points with center curves connecting to $L_i$. Sets $A_2$ and $A_3$ are non-empty by construction and pairwise disjoint by boundary leaf distinctness. Lemma~\ref{connect_two_leaves} implies $\bigcup_{i=2}^k A_i = L_1$. This yields a contradiction: $L_1$ is connected but partitioned into finitely many disjoint non-empty open subsets, which is impossible for a leaf in a lamination.
\end{proof}

This corollary relies fundamentally on Lemma \ref{connect_two_leaves}. Without it, if boundary points existed lacking center curves to other leaves, complementary regions could have more boundary leaves.

\begin{lem}\label{dense}
    Given any $U\subset V$ path-connected complementary region of $\Lambda$, we define $A \subset U$ as the set of points admitting center curves connecting both boundary leaves. Then $A=U$.
\end{lem}

\begin{proof}
    One can see that $A$ is an open set by the continuity of center curves. We claim that the length of center curves in $A$ is uniformly bounded from above. Indeed, it is obvious for points lying in the interstitial regions. We can define a function $\tau:A\rightarrow \mathbb{R}^+$ such that $\tau(x)$ represents the length of the center curve through $x$ connecting two boundary leaves of $U$. This function is clearly continuous. As the gut of $U$ is compact, continuity of $\tau$ gives a uniform upper bound.

    If $A$ does not coincide with $U$, then there exists $y\in \overline{A}\setminus A$ accumulated by $x_n\in A$ as $n\rightarrow \infty$. Since $\tau(x_n)$ is uniformly bounded, the center segments of $x_n$ joining the two boundary leaves of $U$ converge to a center curve through $z$ joining the same boundary leaves with a bounded length. It implies $z\in A$ and thus $A=U$.
\end{proof}

We now establish unique center integrability in complementary regions.

\begin{prop}\label{integ-gut}
    For any $f$-invariant compact lamination $\Lambda\subset \Gamma(f)$ (could be proper) without compact leaves, the bundle $E^c$ is uniquely integrable in every complementary region of $\Lambda$. Consequently, each closed complementary region of $\Lambda$ forms an $I$-bundle, and $\Lambda$ extends canonically to a foliation of $M$ without compact leaves.
\end{prop}

\begin{proof}
    First, we prove unique integrability of $E^c$ in $A\subset U$ for any complementary region $U\subset V$. Take any point $y \in A$. By definition of $A$, there exists a center curve $\gamma$ through $y$ connecting both boundary leaves. The argument in Proposition \ref{integ-boundary} implies that after finitely many iterates, $\gamma$ lies entirely within interstitial regions. Lemma \ref{visit_inters_infty} then forces unique integrability at $y$. Lemma~\ref{dense} establishes unique integrability throughout $U$. The connectivity via center curves of uniformly bounded lengths implies the $I$-bundle of $U$, completing the proof.
\end{proof}

The preceding proposition establishes that when an invariant lamination $\Lambda\subset\Gamma(f)$ contains no compact leaves, each complementary region fibers over a boundary leaf via center segments. Hence, each region is homeomorphic to the product of a boundary leaf with the open interval $(0,1)$. Since $E^c$ is uniquely integrable throughout these regions, $\Lambda$ trivially extends to a foliation without compact leaves on all of $M$. The particular case $\Lambda=\Gamma(f)$ follows immediately.

\subsection{Completion of unique integrability}

We devote this subsection to the proof of Theorem \ref{integrability} for closed manifolds of arbitrary dimension with 1-dimensional center bundle.

\begin{prop}\label{integ-rest}
	The center bundle $E^c$ is uniquely integrable in $\Gamma(f)$.
\end{prop}

\begin{proof}
	For each point $x\in \Gamma(f)$, it suffices to show that both $E^c_+(x)$ and $E^c_-(x)$ are uniquely integrable. Recall that we denote by $\mathcal{C}^+(x)$ and $\mathcal{C}^-(x)$ the collection of integral curves of $E^c_+(x)$ and $E^c_-(x)$, respectively, with the common endpoint $x$. If $x$ is an interior point of $\Gamma(f)$, then both $\mathcal{C}^+(x)$ and $\mathcal{C}^-(x)$ contain small curves intersecting infinitely many plaques of $\Gamma(f)$. Lemma~\ref{integ-accumulated} then implies unique integrability of $E^c$ at $x$.

    If $x$ is a boundary point of $\Gamma(f)$, then there exists a small curve in either $\mathcal{C}^+(x)$ or $\mathcal{C}^-(x)$ intersecting infinitely many plaques of $\Gamma(f)$. We assume this for $\mathcal{C}^+(x)$. Then, by Lemma~\ref{integ-accumulated}, $E^c_+(x)$ is uniquely integrable. As $x$ is a boundary point of $\Gamma(f)$, $\mathcal{C}^-(x)$ contains a small curve lying in a complementary region. Proposition~\ref{integ-boundary} ensures unique integrability for $E^c_-(x)$. Thus $E^c(x)$ is uniquely integrable at all points of $\Gamma(f)$. 
\end{proof}

Now, we are ready to present the proof of Theorem \ref{integrability}.

\begin{proof}[Proof of Theorem \ref{integrability}]
	As explained in the begining of the section, we can assume that the manifold $M$ and the center bundle $E^c$ are both oriented with orientations preserved by $f$. Since $f$ is not accessible, the non-open accessibility classes constitute a non-empty codimension-one compact set $\Gamma(f)\neq \emptyset$ by Proposition \ref{jointint}.
	
	The center bundle $E^c$ is uniquely integrable if $\Gamma(f)=M$ by Proposition \ref{Gamma(f)=M}. It suffices to consider $\emptyset\neq\Gamma(f)\neq M$ as a proper lamination. We obtain the unique integrability in the case where $\Gamma(f)$ contains a compact leaf in Proposition \ref{integ_compactleaf}. In the absence of compact leaves, we obtain the central unique integrability in $\Gamma(f)$ (Proposition~\ref{integ-rest}) and its all complementary regions (Proposition~\ref{integ-gut}). Thus, we complete the proof.
\end{proof}

\subsection{Description on accessibility classes}

This subsection provides a complete characterization of accessibility classes for systems without periodic points, including the proof of Theorem \ref{codim1}.

A foliation is \textit{minimal} if every leaf is dense in $M$. A subset $\Lambda \subset M$ is a \textit{minimal set} of a foliation if it is a compact sublamination where each leaf is dense in $\Lambda$.

\begin{thm}{\cite[Theorem 4.1.3]{HectorHirsch}}\label{hectorhirsch}
    Any codimension-one foliation without compact leaves on a closed manifold has finitely many minimal sets.
\end{thm}

For a partially hyperbolic diffeomorphism with one-dimensional center, we could obtain the following dichotomy on the set of non-open accessibility classes.

\begin{prop}\label{minimal}
	Let $f: M\to M$ be a partially hyperbolic diffeomorphism of a closed manifold with $dim E^c=1$. If $f$ has no periodic points and is not accessible, then either
	\begin{enumerate}
	    \item $\Gamma(f)$ contains infinitely many compact $su$-leaves; or
        \item $\Gamma(f)$ admits a unique minimal set (not necessarily distinct from $\Gamma(f)$).
	\end{enumerate}
\end{prop}

\begin{proof}
	
	If $\Gamma(f)$ contains a compact $su$-leaf, then there must be infinitely many compact leaves since $f$ has no periodic points. Assume that $\Gamma(f)$ contains no compact leaves. By Proposition~\ref{integ-gut}, if $\Gamma(f)$ is a proper lamination, then it trivially extends to a foliation, denoted by $\mathcal{F}$, without compact leaves. When $\Gamma(f)$ is a foliation, we denote by $\mathcal{F}=\Gamma(f)$. By Theorem~\ref{hectorhirsch}, the foliation $\mathcal{F}$ has finitely many minimal sets, hence so does $\Gamma(f)$. The union of all minimal sets in $\Gamma(f)$ is $f$-invariant by the invariance of $\Gamma(f)$. Suppose that there are two distinct minimal sets  $\Lambda\neq \Lambda'\subset\Gamma(f)$. Then, up to a finite iterate, we can assume that both $\Lambda$ and $\Lambda'$ are $f$-invariant. 
	
	By Proposition~\ref{integ-gut}, each closed complementary region of $\Lambda$ forms an $I$-bundle with exactly two boundary leaves. Moreover, the fibers of each $I$-bundle can be taken as center segments. If there is a leaf $L\in \Lambda'$ intersecting an interstitial region of any complementary region of $\Lambda$, then it must intersect all $I$-fibers in interstitial regions by transversality. As any interstitial region becomes arbitrarily thin away from the gut, the leaf $L$ accumulates on some boundary leaf of $\Lambda$. But then $\overline{L}\subset \Lambda'$ has non-empty intersection with $\Lambda$. This is impossible since $\Lambda$ and $\Lambda'$ are disjoint. 
	
	It turns out that $\Lambda'$ lies entirely within a gut piece $D\subset G$ of some complementary region of $\Lambda$. Denote by $\hat{U}$ be the closed complementary region containing $D$. As $\Lambda$ and $\Lambda'$ are both compact, they are uniformly separated; that is, their distance is bounded from below. Then, each boundary leaf of $\Lambda'$ is entirely contained in the gut piece $D$. Note that the union $\Lambda\cup \Lambda'$ is also an $f$-invariant compact lamination. By Proposition~\ref{integ-gut}, each closed complementary region of $\Lambda\cup \Lambda'$ is an $I$-bundle and has exactly two boundary leaves. Since $\Lambda$ and $\Lambda'$ are disjoint, the boundary of $\Lambda\cup\Lambda'$ consists of the boundary leaves of $\Lambda$ and $\Lambda'$. Now, the closed complementary region of $\Lambda\cup \Lambda'$ in $\hat{U}$ has at least three boundary leaves, including two boundary leaves of $\Lambda$ and at least one boundary leaf of $\Lambda'$. Thus, we arrive at a contradiction, which implies that $\Gamma(f)$ has at most one minimal set.
	
	Since $\Gamma(f)$ is a compact lamination without compact leaves, it admits at least one minimal set. It implies that $\Gamma(f)$ admits exactly one minimal set, which could be $\Gamma(f)$ or a proper subset. Hence, we finish the proof.
\end{proof}

Now, we provide our proof of Theorem \ref{codim1}.

\begin{proof}[Proof of Theorem \ref{codim1}]
	Note that it suffices to show the result up a finite lift and a finite iterate. Indeed, any finite iterate of $f$ admits the same accessibility classes as $f$. Any accessibility class in a finite cover projects by the covering map to a subset in an accessibility class of $f$ in $M$. Each minimal set in a finite cover also projects to a minimal set in $M$. Moreover, the absence of periodic points persists under finite iterates and finite covers. Thus, by taking a finite lift and an iterate if necessary, we can assume that $M$ and $E^c$ and both oriented and $f$ preserves their orientations.
    
    Assume that $f$ is not accessible, which means that the set $\Gamma(f)$ is not empty. If $\Gamma(f)$ contains finitely many compact leaves, then each one of them is $f$-periodic. Up to a finite iterate, the restriction of $f$ to any such compact leaf is uniformly hyperbolic, implying the existence of periodic points by the Anosov Closing Lemma (see also Lemma \ref{closinglemma}). Thus $\Gamma(f)$ contains infinitely many compact accessibility classes whenever it contains at least one compact leaf.
	
	Assume that $\Gamma(f)$ has no compact leaves. By Proposition \ref{minimal}, there exists a unique minimal set $\Lambda\subset \Gamma(f)$, necessarily $f$-invariant. The case where $\Lambda=M$ forces $\Gamma(f)$ to be a minimal foliation without compact leaves. If $\Lambda$ is a proper lamination, then by Proposition \ref{integ-gut}, the closed complementary regions of $\Lambda$ are all $I$-bundles by center segments. Moreover, $\Lambda$ trivially extends to a foliation without compact leaves.
	
	If we further assume $NW(f)=M$, then $\Lambda$ must be a minimal foliation and $\Lambda=\Gamma(f)=M$. Otherwise, if $\Lambda$ is a proper subset, then it has finitely many complementary regions with gut pieces. Moreover, each boundary leaf of such complementary region is $f$-periodic. As shown by \cite[Proposition A.5]{08invent}, such boundary leaves of $\Lambda$ contain dense periodic points, which contradicts the absence of periodic points. 
	
	Hence, we finish the proof.
\end{proof}

\section{Accessibility}\label{section4}

In this section, we are going to show the accessibility of a partially hyperbolic diffeomorphism without periodic points in a closed three-dimensional manifold, see Theorem \ref{accessible}. Before doing that, we shall introduce the completeness of center-stable foliations for non-accessible partially hyperbolic diffeomorphisms without periodic points whenever the center bundle is one-dimensional.

\subsection{Completeness of center stable foliation}

Let $M$ be a closed manifold and $f:M\rightarrow M$ be a non-accessible partially hyperbolic diffeomorphism so that it has no periodic points and the center bundle $E^c$ is one-dimensional. As shown by Theorem \ref{integrability}, the center bundle $E^c$ is uniquely integrable. Denote by $\mathcal{F}^c$ the unique $f$-invariant foliation tangent to $E^c$. In particular, there exist $f$-invariant center-stable and center-unstable foliations, denoted by $\mathcal{F}^{cs}$ and $\mathcal{F}^{cu}$, respectively. For any set $S\subset M$, define $W^{\sigma}(S):=\bigcup\limits_{x\in S}W^{\sigma}(x)$ $(\sigma=s,c, u)$, where $W^{\sigma}(x)$ is the immersed manifold tangent to $E^{\sigma}$ through $x$.

\begin{defn}
	The center-stable foliation $\mathcal{F}^{cs}$ is \emph{complete} if $\mathcal{F}^{cs}(x)=W^{s}(W^c(x))$ for every point $x\in M$.
\end{defn}

This concept, introduced in \cite{BW05}, implies that every center-stable leaf is a topological product of center and stable leaves. Equivalently, the foliation $\mathcal{F}^{cs}$ is complete if and only if $\mathcal{F}^{cs}(x)=W^s(W^c(x))=W^c(W^s(x))$ for any $x\in M$, see \cite[Section 3]{FU1}. When $\mathcal{F}^{cs}$ is not complete, there exists a point $x\in M$ such that $W^s(W^c(x))$ is a proper subset of the leaf $\mathcal{F}^{cs}(x)$. In such case, there must be a point $y\in \mathcal{F}^{cs}(x)$ and a center curve $\gamma : [0, 1]\rightarrow M$ satisfying $$ \gamma (\left[0, 1\right))\subset W^s(W^c(x)) \quad  and \quad  \gamma(1)= y \notin W^s(W^c(x)).$$ We call $y$ an \emph{accessible boundary point} of $W^s(W^c(x))$.

\begin{lem}\label{complete-lem}
    Let $f:M\to M$ be a non-accessible partially hyperbolic diffeomorphism without periodic points and with one-dimensional center bundle $E^c$. If $\mathcal{F}^{cs}(x)$ is not complete for some $x\in M$, then for each accessible boundary point $y$ of $W^s(W^c(x))$, the $\omega$-limit set $\omega(y)$ belongs to $\Gamma(f)$.
\end{lem}
\begin{proof}
    If $y$ is contained in $\Gamma(f)$, then the limit set $\omega(y)$ belongs to $\Gamma(f)$ by the invariance and compactness of $\Gamma(f)$. It suffices to consider the case where $\Gamma(f)$ is a proper lamination and $y$ lies in the complement of $\Gamma(f)$. We will discuss separately the situations whether $\Gamma(f)$ contains a compact leaf or not.

    Let $\delta>0$ be a small number such that for any point $x\in M$, the $\delta$-neighborhood of the leaf $W^c(x)$ in $\mathcal{F}^{cs}(x)$ is contained in $W^{s}(W^c(x))$. Such a $\delta$ exists by transversality of the center and stable foliations; see Lemma~\ref{lem1}. Since the center curve $\gamma$ can be taken as short as desired, we assume that the length of $\gamma$ is less than $\frac{\delta}{16}$. Note that $f^i(y)$ is at least $\delta $-away from $W^c(f^i(x))$ for any $i\in \mathbb{Z}$ by the choice of $\delta$. Otherwise $f^i(y)$ would be contained in $W^s(W^c(f^i(x)))$ and then $y$ would be contained in $W^s(W^c(x))$. Moreover, as $\gamma(0) \in W^s(W^c(x))$, there is $N\in \mathbb{N}$ such that for any $m\geq N$, the distance from $f^m(\gamma(0))$ to $W^c(f^m(x))$ in the leaf $\mathcal{F}^{cs}(f^m(x))$ is smaller than $\frac{\delta}{2}$. This implies that the length of $f^m(\gamma)$ is greater than $\frac{\delta }{2}$ for any $m\geq N$. 
    
    When $\Gamma(f)$ contains a compact leaf, it contains infinitely many compact leaves by Theorem~\ref{codim1}. Denote by $\mathcal{K}\subset\Gamma(f)$ the set of all compact leaves, forming a compact invariant lamination by Theorem~\ref{Hae62}. Denote by $R$ the maximal connected complementary region of $\mathcal{K}$ containing $y$, which is non-empty since $y$ lies in the complement of $\Gamma(f)$. By taking a smaller $\delta>0$ if necessary, we assume the center curve $\gamma$ through $y$ lies entirely within $R$ of length less than $\frac{\delta}{16}$. The argument above establishes that the length of $f^m(\gamma)\subset f^m(R)$ is greater than $\frac{\delta}{2}$ for any $m\geq N$. Lemma~\ref{lem1} then ensures a constant $\rho>0$ dominating the radius of $f^m(R)$ for each $m\geq N$. This contradicts with the compactness of $M$.
    
    Thus, we conclude $y\in \mathcal{K}$ whenever there exists a compact accessibility class, and thus $\omega(y)\subset \mathcal{K}\subset\Gamma(f)$. Now, let us consider the case without compact leaves.
    
    By the absence of compact accessibility class, the complementary regions of the lamination $\Gamma(f)$ consist of finitely many gut pieces and some interstitial regions. Assume that $y$ is contained in a complementary region $U$ of $\Gamma(f)$. By Proposition \ref{integ-gut}, all complementary regions of $\Gamma(f)$ are $I$-bundles by center segments. The point $y$ admits a center segment $c$ whose two endpoints are contained in two boundary leaves of $\hat{U}$ and the interior points are contained in $U$. Let $y^+$ be one of the endpoints of $c$. By Lemma \ref{suplaque}, there are finitely many $su$-plaques in the boundary of each gut. Lemma \ref{closinglemma} ensures that the iterations of $y^+$ are all contained in distinct $su$-plaques. Note that there are finitely many guts due to the compactness of the manifold. Then there is $k\in \mathbb{N}$ such that for any $n\in\mathbb{Z}$ with $|n|\geq k$, the point $f^n(y^+)$ is contained in the boundary of the interstitial regions. As $\Gamma(f)$ and its complement are $f$-invariant, the center segment $f^n(c)$ is contained in the interstitial regions for $|n|\geq k$ except its two endpoints. In particular, each point $f^n(y)$ is contained in the interstitial regions.
	
	Suppose that there is a point $p\in \omega(y)$ lying in the complement of $\Gamma(f)$. Then it must be contained in an interstitial region, denoted by $O$, since $f^n(y)$ cannot intersect any gut for every $|n|\geq k$. The center curve through $p$ intersects the boundary of $O$ in a point, denoted by $p^+\in \Gamma(f)$. Pick a small $\epsilon_0>0$ as in Lemma \ref{closinglemma}. Recall that $E^c$ is uniquely integrable by Theorem \ref{integrability}. By Lemma~\ref{lem1} and openness of $O$, there exists $\tau>0$ such that $B_{\tau}(p)\subset O$ and any point $z\in B_{\tau}(p)$ admits a center curve intersecting $W^s_{\epsilon_0}(W^u_{\epsilon_0}(p^+))$ in a unique point. Let $\{n_i\}_{i\in \mathbb{N}}$ be a sequence of integers such that $f^{n_i}(y)$ converges to $p$ as $i\rightarrow \infty$. There exists $N\in \mathbb{N}$ such that for any $i\geq N$, the point $f^{n_i}(y)$ is contained in $B_{\tau}(p)$. Thus, for any $i\geq N$, $f^{n_i}(y)$ admits a center curve intersecting $W^s_{\epsilon_0}(W^u_{\epsilon_0}(p^+))$ in a unique point. Taking $i> j\geq N$, we denote by $y_i, y_j\in W^s_{\epsilon_0}(W^u_{\epsilon_0}(p^+))$ two points in center curves through $f^{n_i}(y), f^{n_j}(y)$, respectively. Since the boundary of $\Gamma(f)$ is $f$-invariant, we deduce $f^{n_i-n_j}(y_j)=y_i$. Then, Lemma \ref{closinglemma}, produces a periodic point in $W^s_{\epsilon_0}(W^u_{\epsilon_0}(p))$, which is absurd. 
    
    Therefore, the limit set $\omega(y)$ must be a subset of $\Gamma(f)$ and we finish the proof.
\end{proof}

We mention that \cite[Lemma 6.3]{2020Seifert} establishes the completeness of both center-stable and center-unstable foliations for partially hyperbolic diffeomorphisms with $\Gamma(f)=M$ and no periodic points in closed 3-manifolds. Here, we generalize this result to non-accessible diffeomorphisms in higher dimensions with one-dimensional center bundles:

\begin{prop}\label{complete}
    Let $f:M\to M$ be a non-accessible partially hyperbolic diffeomorphism without periodic points and with one-dimensional center bundle $E^c$. Then, both center-stable ($\mathcal{F}^{cs}$) and center-unstable ($\mathcal{F}^{cu}$) foliations are complete.
\end{prop}

\begin{proof}
	Since $f$ is not accessible, the set $\Gamma(f)$ is either a foliation or a proper lamination. We only present our proof in the case where $\Gamma(f)$ is a lamination since the foliation case follows in a simpler way. Suppose to the contrary that $\mathcal{F}^{cs}$ is not complete, implying $\mathcal{F}^{cs}(x) \neq W^{s}(W^c(x))$ for some $x\in M$. Let $y\in \mathcal{F}^{cs}(x)$ be an accessible boundary point of $W^{s}(W^c(x))$, and $\gamma : [0, 1]\rightarrow M$ be a center curve satisfying that $ \gamma(1)= y$ and $\gamma (\left[0, 1\right))\subset W^s(W^c(x))$.

    Let $\delta>0$ be given as in the proof of Lemma~\ref{complete-lem}. Replacing $\gamma$ by a shorter curve if necessary, the argument in Lemma~\ref{complete-lem} shows the existence of $N\in \mathbb{N}$ such that the length of $f^m(\gamma)$ is greater than $\frac{\delta}{2}$ for each $m\geq N$.
	
	As the manifold is compact and $E^c$ is one-dimensional, by taking a smaller $\delta$ if necessary, we can order all $su$-plaques of $\Gamma(f)$ in each coordinate cube of radius at most $\delta$. Take a point $z\in \omega(y)$. Lemma~\ref{complete-lem} implies $z\in\Gamma(f)$. Since the manifold is compact and $f$ has no periodic points, there are three iterations of $z$ in an open set with order $\omega_1<\omega_3<\omega_2$ such that $dist_c(\omega_3, \omega_i)<\frac{\delta}{8}$, $i=1,2$, where $dist_c$ is the infimum of the length of center curves joining corresponding $su$-plaques. Take two integers $n_1, n_2\in \mathbb{N}$ such that $dist_c(f^{n_i}(y), \omega_i)<\frac{\delta}{16}$, $i=1, 2$. Note that $f^{n_i}(y)$, $i=1,2$, are also accessible boundary points. 
	
		\begin{figure}[htb]	
		\centering
		\includegraphics{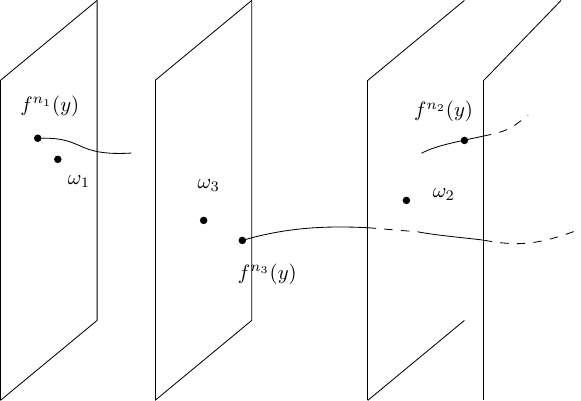}\\
		\caption{The $su$-plaques through $\omega_i$ and center curves through $f^{n_i}(y)$ in a coordinate cube}
		
	\end{figure}
	
	Suppose that $y$ is not contained in $\Gamma(f)$. Then any iteration of $y$ cannot be contained in $\Gamma(f)$ by the invariance of $\Gamma(f)$. We can take two center curves $c_1$ and $c_2$ starting at $f^{n_1}(y)$ and $f^{n_2}(y)$, respectively, as given in the definition of accessible boundary points such that they are disjoint with $\Gamma(f)$. So, there exists $N_i\in \mathbb{N}$ such that for any $n\geq N_i$, the length of $f^n(c_i)$ is greater than $\frac{\delta}{2}$ for $i=1,2$. Pick an integer $n_3\in \mathbb{N}$ such that $n_3-n_i\geq N_i$ and $dist_c(f^{n_3}(y), \omega_3)< \frac{\delta}{16}$ for $i=1,2$. Then, for either $i=1$ or $i=2$, the center curve $f^{n_3-n_i}(c_i)$ intersects the $su$-plaque through the point $\omega_i$. By the invariance of $\Gamma(f)$ and our choice of $c_i$, the center curve $f^{n_3-n_i}(c_i)$ is disjoint with $\Gamma(f)$. This is a contradiction as $\omega_i\in \omega(y)$ is contained in $\Gamma(f)$. Thus, we have $y\in\Gamma(f)$. 
	
	Now, we take a center curve $c_i$ of length less than $\frac{\delta}{16}$ satisfying the definition of accessible boundary point such that it is bounded by $f^{n_i}(y)$ and some other point in $\Gamma(f)$ for $i=1,2$. This is achievable for large $n_i$ when $c_i$ intersects the complement of $\Gamma(f)$ as the interstitial regions can be made arbitrarily thin. However, the interior of the center curve $c_i$ must intersect infinitely many $su$-plaques of $\Gamma(f)$. Otherwise, one can take $c_i$ with interior lying in the complement of $\Gamma(f)$ and deduce a contradiction by applying the argument above. For each $i=1,2$, there is an integer $N_i\in \mathbb{N}$ so that for any $n\geq N_i$, the length of $f^{n}(c_i)$ is greater than $\frac{\delta}{2}$. Take $n_3\in \mathbb{N}$ such that $n_3-n_i\geq N_i$ and $dist_c(f^{n_3}(y), \omega_3)<\frac{\delta}{16}$. Then the center curve $f^{n_3-n_i}(c_i)$ has length greater than $\frac{\delta}{2}$. Under the distance $dist_c$, the center curve $c_i$ is entirely contained in the $\frac{\delta}{2}$-neighborhood of $f^{n_3}(y)$ for $i=1,2$. Then for either $i=1$ or $i=2$, projecting along $su$-plaques, we obtain a map $f^{n_i-n_3}$ acting contractively in an interval $I$ whose two endpoints correspond to two $su$-plaques in $\Gamma(f)$. By the invariance and compactness of $\Gamma(f)$, there is an $su$-plaque in $\Gamma(f)$ containing some point $p$ and its iteration $f^{n_i-n_3}(p)$. Lemma \ref{closinglemma} implies the existence a periodic point of $f$, contradicting with the assumption. 
	
	Hence, we complete the proof.
\end{proof}

\subsection{Proof of Theorem \ref{accessible}}

From now on, we restrict consideration to three-dimensional manifolds $M$. After possibly passing to a finite cover and iterate, we assume $M$ and the bundles $E^\sigma$ ($\sigma = c,s,u$) are orientable, with orientations preserved by $f$. This assumption will not affect our conclusions since accessibility of $f$ derives from that of its finite iterates in covers.

We now start the proof of Theorem \ref{accessible}.

Suppose $f:M\rightarrow M$ is a non-accessible partially hyperbolic diffeomorphism without periodic points. If $\Gamma(f)$ contains a compact accessibility class, then $\pi_1(M)$ is virtually solvable by Theorem \ref{maptori}. When no compact accessibility classes exist, Theorem \ref{codim1} implies either that $\Gamma(f)$ is itself a minimal foliation or extends to a foliation without compact leaves.
By Proposition \ref{minimal}, $\Gamma(f)$ possesses a unique minimal set denoted $\Lambda^{su}$, possibly equal to $\Gamma(f)$. Theorem \ref{integrability} establishes unique integrability of the center bundle $E^c$. Consequently, there exist invariant foliations $\mathcal{F}^{cs}$ and $\mathcal{F}^{cu}$ tangent to $E^{cs}$ and $E^{cu}$ respectively, neither of which admits compact leaves by Theorem \ref{nocompactleaf}. Proposition \ref{complete} confirms both $\mathcal{F}^{cs}$ and $\mathcal{F}^{cu}$ are complete.

Unique integrability of $E^c$ implies the existence of a unique center foliation $\mathcal{F}^c$. Let $\widetilde{\mathcal{F}}^{\sigma}$ ($\sigma = c, cs, cu$) denote the lifted foliations in the universal cover $\widetilde{M}$, and $\widetilde{\Lambda}^{su}$ the lifted lamination.

A codimension-one foliation is $\mathbb{R}$-covered when the leaf space of its lifted foliation in $\widetilde{M}$ is homeomorphic to $\mathbb{R}$. Equivalently, a codimension-one lamination is \emph{$\mathbb{R}$-covered} if the leaf space of its lifted lamination forms a Hausdorff separable one-dimensional manifold. For proper laminations, this leaf space is not necessarily simply-connected. When a codimension-one lamination $\Lambda$ extends trivially to a foliation $\mathcal{F}$, $\Lambda$ is $\mathbb{R}$-covered precisely when $\mathcal{F}$ is by definition.

The following proposition is summarized from \cite[Section 3]{FU1}.

\begin{prop}\label{1point}
	Let $f:M^3\rightarrow M^3$ be a partially hyperbolic diffeomorphism of a closed 3-manifold with complete center-stable and center-unstable foliations $\mathcal{F}^{cs}$, $\mathcal{F}^{cu}$. Then each leaf of $\mathcal{F}^{cs}$ and $\mathcal{F}^{cu}$ is either a plane or a cylinder. The lifted foliations $\mathcal{\widetilde{F}}^{cs}$ and $\mathcal{\widetilde{F}}^{cu}$ are also complete. If $\Lambda^{su}$ is a minimal $su$-lamination without compact leaves, then each leaf of $\mathcal{\widetilde{F}}^c$ intersects every leaf of $\widetilde{\Lambda}^{su}$ in exactly one point. In particular, the lamination $\Lambda^{su}$ is $\mathbb{R}$-covered.
\end{prop}

We recall the following result which does not rely on the dynamics.

\begin{prop}\cite[Proposition 5.7]{FP_hyperbolic}
	Let $M$ be an irreducible closed 3-manifold whose fundamental group is not (virtually) solvable. Assume that $\Lambda\subset M$ is a minimal lamination without compact leaves such that each closed complementary region is an $I$-bundle. Then the leaves of $\Lambda$ are uniformly Gromov hyperbolic.
\end{prop}

As a direct consequence of the preceding proposition, we are able to obtain the Gromov hyperbolicity of leaves of $\Lambda^{su}$.

\begin{cor}\label{hyperbolicleaf_generalcase}
	The lamination $\Lambda^{su}$ comprises uniformly Gromov hyperbolic leaves.
\end{cor}

In the universal cover, each leaf $L \in \widetilde{\Lambda}^{su}$ is identified with an open Poincaré disk. We compactify $L$ by adding an ideal circle at infinity, denoted $\partial_{\infty}L$. Intersecting $\widetilde{\mathcal{F}}^{cs}$ with $L$ yields a one-dimensional foliation $\widetilde{\Lambda}^s_L$. Similarly, intersecting $\widetilde{\mathcal{F}}^{cu}$ with $L$ defines $\widetilde{\Lambda}^u_L$. By Theorem~\ref{codim1}, $\Lambda^{su}$ trivially extends to a foliation $\mathcal{F}$ by non-compact leaves. The leaves of $\mathcal{F}$ exhibit the same geometry as those of $\Lambda$. Precisely, each leaf of $\mathcal{F}$ exhibits Gromov hyperbolicity and each lifted leaf can be compactified by an ideal circle. Let $\mathcal{A} := \bigcup_{L \in \widetilde{\Lambda}^{su}} \partial_{\infty}L$ and $\mathcal{A}' := \bigcup_{L \in \widetilde{\mathcal{F}}} \partial_{\infty}L$ denote the union of ideal circles of $\widetilde{\Lambda}^{su}$ and $\widetilde{\mathcal{F}}$, respectively. Following \cite[Section 4]{FU1}, \cite{Calegari00, Fenley02}, we endow $\mathcal{A}$ and $\mathcal{A}'$ with a topology by identifying it with unit tangent circles at arbitrary basepoints in each leaf. Consequently, $\widetilde{M} \cup \mathcal{A}'$ is homeomorphic to a solid cylinder.

As analyzed in \cite[Section 4]{FU1}, we examine the asymptotic behavior of $\widetilde{\Lambda}^s_L$ leaves for each $L \in \widetilde{\Lambda}^{su}$. Note that each deck transformation fixing some leaf of $\widetilde{\Lambda}^{su}$ is a hyperbolic M\"{o}bius transformation with exactly two fixed points in the ideal boundary of the leaf. Moreover, Theorem \ref{rosenberg} guarantees a non-trivial deck transformation fixing some leaf of $\widetilde{\Lambda}^{su}$. We will mainly focus on the leaves of $\widetilde{\Lambda}^{su}$.

Applying arguments from \cite[Section 4,5]{FU1}, we obtain these key conclusions:

\begin{prop}
	For any leaf $L\in \widetilde{\Lambda}^{su}$, each ray of any leaf of $\widetilde{\Lambda}^s_L$ accumulates in a single ideal point in $\partial_{\infty}L$. Let $S^{\infty}_L\subset \partial_{\infty}L$ be the set of ideal limit points of $\widetilde{\Lambda}^s_L$. With respect to the given topology of $\mathcal{A}$, we have two possibilities:
	\begin{itemize}
		\item If the set $S^{\infty}_L$ is not dense in $\partial_{\infty}L$ for some $L\in \widetilde{\Lambda}^{su}$, then the foliation $\Lambda^{su}$ is non-uniform. Moreover, if $\Lambda^{su}$ is a foliation, then it is the weak stable foliation of a flow conjugate to a suspension Anosov flow, and the underlying manifold is a solvmanifold.
		\item If the set $S^{\infty}_L$ is dense in $\partial_{\infty}L$ for every $L\in \widetilde{\Lambda}^{su}$, then every stable leaf $l\in \widetilde{\Lambda}^s_L$ has two distinct ideal points. Equivalently, the leaf space of $\widetilde{\Lambda}^s_L$ in $L$ is Hausdorff, and all leaves of $\widetilde{\Lambda}^s_L$ are uniform quasi-geodesics. Each leaf of $\widetilde{\Lambda}^{su}$ forms a weak quasi-geodesic fan for $\widetilde{\Lambda}^s$. Each leaf of $\Lambda^{su}$ is either a cylinder or a plane. Furthermore, the ideal points of stable leaves of $\widetilde{\Lambda}^s_L$ vary continuously.
	\end{itemize}
	
\end{prop}

By collapsing complementary regions of $\Lambda^{su}$ along center segments, we construct a new manifold $M_0 = M / \sim$ where $\sim$ identifies points on the same $I$-fiber within any complementary region. In $M_0$, $\Lambda^{su}$ becomes a minimal foliation. Crucially, since $\pi_1(M)$ is not virtually solvable, neither is $\pi_1(M_0)$. In fact, $M_0$ is homeomorphic to $M$. This collapsing preserves the completeness of $\mathcal{F}^{cs}$ and $\mathcal{F}^{cu}$ because the $I$-bundle structure in closed complementary regions is defined by center fibers. 

The first scenario cannot occur because $\pi_1(M)$ is not virtually solvable. In the second case, Theorem \ref{rosenberg} implies $\Lambda^{su}$ contains at least one cylindrical leaf. Consequently, there exists a non-trivial deck transformation $h$ fixing a leaf $F \in \widetilde{\Lambda}^{su}$. Due to the leafwise weak quasi-geodesic fan structure, all leaves of $\widetilde{\Lambda}^s_F$ accumulate at a unique ideal point $\xi_F \in \partial_{\infty}F$. This $\xi_F$ is one of two fixed points of $h$; denote the other by $\eta \in \partial_{\infty}F$. By denseness of $S^\infty_L$ and continuity of ideal points, there exists a stable leaf $l \in \widetilde{\Lambda}^s_F$ asymptotic to both $\xi_F$ and $\eta$. Let $l^*$ be the associated geodesic in $F$ connecting $\xi_F$ and $\eta$, which forms the axis of $h$. Since $l$ is quasi-geodesic, it lies entirely within some $K$-neighborhood of $l^*$ for a uniform constant $K>0$. As $h$ fixes ideal points $\xi_F$ and $\eta$, its action preserves this neighborhood, so $h^i(l)$ also lies within the $K$-neighborhood of $l^*$ for any $i\in \mathbb{Z}$. Thus, within the closure of this $K$-neighborhood, there exists an $h$-invariant leaf of $\widetilde{\Lambda}^s_F$. This implies the existence of a closed stable leaf for $f$ in $M$, which is impossible.

Hence, we complete the proof of Theorem \ref{accessible}.

\subsection{Proof of Corollary~\ref{accessible=nosutorus}}

Now, we give our proof of Corollary \ref{accessible=nosutorus} using Theorem \ref{accessible}.

\begin{proof}[Proof of Corollary \ref{accessible=nosutorus}]
    Accessibility clearly implies the absence of $su$-tori. Assume that $f: M\to M$ is a non-accessible partially hyperbolic diffeomorphism of a closed 3-manifold without periodic points. By Theorem \ref{accessible}, after passing to a finite cover, the fundamental group $\pi_1(M)$ is solvable. According to \cite{2008nil} and \cite[Theorem 2.5]{Ham17CMH}, either there exists a 2-torus tangent to $E^s \oplus E^u$ or $f$ lies in the homotopy class of an Anosov automorphism on $\mathbb{T}^3$. In the latter case, $f$ must possess periodic points, contradicting our assumption. Therefore, an $su$-torus must exist.

    Assuming the existence of $su$-tori, let $\Lambda$ denote the collection of all such tori. Theorem \ref{Hae62} establishes that $\Lambda$ forms a compact $f$-invariant lamination. No leaf in $\Lambda$ can be periodic under $f$, because otherwise an iterate of $f$ restricted to a periodic $su$-torus would be Anosov and yield periodic points. Theorem~\ref{integrability} ensures unique integrability of $E^c$, implying $f$ is dynamically coherent. Since $\pi_1(M)$ is solvable, Theorem~\ref{classification_sol} guarantees an $f$-periodic compact center curve $\gamma$. Since $f$ is $C^2$ and $\gamma$ is 2-normally hyperbolic, $\gamma$ forms a $C^2$ circle. After adjusting by the period, we may assume $\gamma$ is $f$-invariant. The restriction $f|_{\gamma}$ has no periodic points, resulting in an irrational rotation number. As $f|_{\gamma}$ is $C^2$, Denjoy's theorem ensures every orbit is dense. Note that $\Lambda \cap \gamma$ is $f|_{\gamma}$-invariant. Consequently, $\Lambda$ must have trivial complement and therefore constitutes a foliation of $M$.

    Thus, we complete the proof.
\end{proof}

\section{Complement of accessibility}\label{consequence}

In this section, we examine further dynamical properties of non-accessible partially hyperbolic diffeomorphisms without periodic points. We characterize which partially hyperbolic diffeomorphisms simultaneously exhibit non-accessibility and absence of periodic points, along with typical dynamical properties for such systems.

\subsection{Classification}

This subsection focuses on a closed three-dimensional manifold $M$ and a partially hyperbolic diffeomorphism $f:M\rightarrow M$ without periodic points and with non-accessibility. We devote this subsection to proving Theorem \ref{classification}.

Recall that Theorem~\ref{classification_sol} classifies dynamically coherent partially hyperbolic diffeomorphisms on closed 3-manifolds with solvable fundamental group. The dynamical coherence assumption in this classification is essential. Indeed, examples exist on $\mathbb{T}^3$ admitting 2-tori tangent to the center-stable or center-unstable bundle \cite{2016example}; these are neither dynamically coherent nor represented in any class listed in Theorem~\ref{classification_sol}. Moreover, for such manifolds, the existence of $cs$- or $cu$-tori is equivalent to non-dynamical coherence \cite{HP15, 2015Center}.


Now, let us present our proof.

\begin{proof}[Proof of Theorem \ref{classification}]
    Let $f:M\rightarrow M$ be a partially hyperbolic diffeomorphism of a closed 3-manifold without periodic points, and assume $f$ is non-accessible. By Theorem \ref{integrability}, the center bundle $E^c$ is uniquely integrable, which implies $f$ is dynamically coherent. As shown in Theorem \ref{accessible}, the ambient manifold $M$ has virtually solvable fundamental group. Then, up to a finite lift and iterate, $f$ is leaf conjugate to one of three diffeomorphisms listed in Theorem \ref{classification_sol}. The last two cases are desired, so we need only eliminate the first case. In this excluded scenario, the manifold $M$ is the 3-torus and $f$ is homotopic to an Anosov automorphism. But this would force periodic points, contradicting our assumption.
    Therefore, we finish the proof of Theorem \ref{classification}.
\end{proof}

\subsection{Plaque expansiveness}

Given a center foliation $\mathcal{F}^c$ for a partially hyperbolic diffeomorphism $f: M\rightarrow M$, a sequence $x_n$ ($n\in \mathbb{Z}$) is an \emph{$\epsilon$-pseudo orbit along $\mathcal{F}^c$} if $f(x_n) \in \mathcal{F}^c_{\epsilon}(x_{n+1})$ for every $n\in \mathbb{Z}$. The diffeomorphism $f$ is \emph{plaque expansive} on $\mathcal{F}^c$ (equivalently, $\mathcal{F}^c$ is plaque expansive) if there exists $\epsilon>0$ such that for any two $\epsilon$-pseudo orbits $x_n$, $y_n$ along $\mathcal{F}^c$ satisfying $d(x_n, y_n) < \epsilon$ for all $n\in \mathbb{Z}$, we have $x_0 \in \mathcal{F}^c_{loc}(y_0)$.

Whether every dynamically coherent partially hyperbolic diffeomorphism is plaque expansive remains a long-standing open problem. Partial resolutions exist for $C^1$ normally hyperbolic foliations \cite{HPS77}, $C^0$ normally hyperbolic foliations with bi-Lyapunov stable actions \cite{HHU07survey}, and quasi-isometric actions \cite{Berger13}.

A partially hyperbolic diffeomorphism $f: M\rightarrow M$ has a \emph{quasi-isometric action along a center foliation} if there exist an $f$-invariant center foliation $\mathcal{F}^c$ and constants $r, R>0$ such that $f^n(\mathcal{F}^c_{r}(x)) \subset \mathcal{F}^c_R(f^n(x))$ for all $x\in M$ and $n\in \mathbb{Z}$. See \cite{Feng24} for equivalent definitions and further analysis of such systems.

Uniqueness of invariant center foliations remains unresolved even when plaque expansiveness holds. Here, we identify a class of partially hyperbolic systems possessing both a unique invariant center foliation and plaque expansiveness.

\begin{thm}\label{plaqueexpansivity}
	Let $f:M\rightarrow M$ be a $C^1$ partially hyperbolic diffeomorphism of a closed 3-manifold without periodic points. Then, either
	\begin{itemize}
		\item $f$ is accessible; or
		\item $f$ is plaque expansive on the unique center foliation $\mathcal{F}^c$.
	\end{itemize}
	Moreover, there is a $C^1$-neighborhood $U$ of $f$ such that either
	\begin{itemize}
		\item all partially hyperbolic diffeomorphisms in $U$ are accessible; or
		\item every partially hyperbolic diffeomorphism $g\in U$ admits a uniquely integrable center bundle $E^c_g$, is plaque expansive on its unique center foliation, and is leaf conjugate to $f$. In particular, every such $g$ is dynamically coherent. 
	\end{itemize}
	
\end{thm}

This result is motivated by a question of Shaobo Gan.

Now, let us state our proof.
\begin{proof}
Let $f: M\rightarrow M$ be a $C^1$ partially hyperbolic diffeomorphism of a closed 3-manifold without periodic points, and assume $f$ is non-accessible. Theorem \ref{integrability} establishes unique integrability of the center bundle, yielding a unique center foliation $\mathcal{F}^c$. By Theorem \ref{classification}, up to a finite cover and iterate, $f$ becomes leaf-conjugate to either a skew product over a linear Anosov automorphism of $\mathbb{T}^2$, or the time-one map of a suspension Anosov flow. 

In both cases, $f$ possesses a quasi-isometric action along the center foliation $\mathcal{F}^c$ (see \cite{Martinchich23,Feng24}). This property persists prior to taking covers or iterates. Consequently, $f$ exhibits plaque expansiveness on its unique center foliation $\mathcal{F}^c$. According to \cite[Theorem 7.4]{HPS77}, there exists a $C^1$-neighborhood $U_1$ of $f$ where every $g \in U_1$ admits a unique plaque expansive center foliation and is leaf-conjugate to $f$.

When $f$ is accessible, \cite{Didier} guarantees a $C^1$-neighborhood $U_2$ of $f$ such that every $g \in U_2$ remains accessible. Hence, we complete the proof by taking $U = U_1 \cap U_2$.
\end{proof}


We have the following corollary in higher dimensions.
\begin{cor}
	Let $f: M\rightarrow M$ be a partially hyperbolic diffeomorphism with $dim E^c=1$ and no periodic points. Then, we have either
	\begin{itemize}
		\item $f$ is accessible; or
		\item the set $\Gamma(f)$ is non-empty and $E^c$ is uniquely integrable to the center foliation $\mathcal{F}^c$. If $\mathcal{F}^c$ is a compact center foliation, then it is uniformly compact and plaque expansive.
	\end{itemize}
\end{cor}
\begin{proof}
	This is a direct consequence of Theorem \ref{integrability}, Theorem \ref{codim1}, Proposition \ref{complete}, \cite[Theorem 1.2]{Carrasco15}, and \cite[Theorem 4.6]{Carrasco_thesis}.
\end{proof}



\bibliographystyle{alpha}
\bibliography{ref}

\end{document}